\documentclass[leqno,11pt]{amsart}

\usepackage{amssymb, amsmath}
\def\refer#1{~\ref{#1}}
\def\refeq#1{~(\ref{#1})}
\def\ccite#1{~\cite{#1}}

\def\longformule#1#2{
\displaylines{ \qquad{#1} \hfill\cr \hfill {#2} \qquad\cr } }
\def\inte#1{
\displaystyle\mathop{#1\kern0pt}^\circ }
\def\sumetage#1#2{
\sum_{\scriptstyle {#1}\atop\scriptstyle {#2}} }

\def\supetage#1#2{
\sup_{\scriptstyle {#1}\atop\scriptstyle {#2}} }

\let\al=\alpha

\let\e=\varepsilon

\let\lam=\lambda

\let\s=\sigma
\let\f=\phi

\let\wh=\widehat

\def\cF{{\mathcal F}}

\def\cI{{\mathcal I}}

\def\cS{{\mathcal S}}

\def\with{\quad\hbox{with}\quad}
\def\and{\quad\hbox{and}\quad}

\def\virgp{\raise 2pt\hbox{,}}
\def\cdotpv{\raise 2pt\hbox{;}}

\def\eqdefa{\buildrel\hbox{{\rm \footnotesize def}}\over =}

\def\C{\mathop{\mathbb  C\kern 0pt}\nolimits}
\def\DD{\mathop{\mathbb  D\kern 0pt}\nolimits}
\def\K{\mathop{\mathbb  K\kern 0pt}\nolimits}
\def\N{\mathop{\mathbb  N\kern 0pt}\nolimits}
\def\Q{\mathop{\mathbb  Q\kern 0pt}\nolimits}
\def\R{\mathop{\mathbb  R\kern 0pt}\nolimits}
\def\SS{\mathop{\mathbb  S\kern 0pt}\nolimits}
\def\ZZ{\mathop{\mathbb  Z\kern 0pt}\nolimits}
\def\T{\mathop{\mathbb T\kern 0pt}\nolimits}
\def\P{\mathop{\mathbb\kern 0pt}\nolimits}

\newcommand{\ds}{\displaystyle}

\def\dive{\mathop{\rm div}\nolimits}

\newcommand{\beq}{\begin{equation}}
\newcommand{\eeq}{\end{equation}}
\newcommand{\ben}{\begin{eqnarray}}
\newcommand{\een}{\end{eqnarray}}
\newcommand{\beno}{\begin{eqnarray*}}
\newcommand{\eeno}{\end{eqnarray*}}

\newtheorem{theo}{Theorem}
\newtheorem{lemma}{Lemma}[section]

\newtheorem{corol}{Corollary}[section]
\newtheorem{prop}{Proposition}[section]

\def\vp{{\underline v}}
\def\presspO{{{\underline p}_0}}
\def\presspun{{{\underline p}_1}}
\def\wp{{\underline w}}
\def\wpe{{\underline w}^{\e}}
\def\vapp{v_{app}^\e}

\begin{document}

\title[Large, global solutions to the  Navier-Stokes  equations]{Large, global solutions to the 
 Navier-Stokes  equations,   slowly varying in one direction}

\author[J.-Y. Chemin]{Jean-Yves  Chemin}
\address[J.-Y. Chemin]%
{ Laboratoire J.-L. Lions UMR 7598\\ Universit{\'e} Paris VI\\
175, rue du Chevaleret\\ 75013 Paris\\FRANCE }
\email{chemin@ann.jussieu.fr }
\author[I. Gallagher]{Isabelle Gallagher}
\address[I. Gallagher]%
{ Institut de Math{\'e}matiques de Jussieu UMR 7586\\ Universit{\'e} Paris VII\\
175, rue du Chevaleret\\ 75013 Paris\\FRANCE }
\email{Isabelle.Gallagher@math.jussieu.fr}

\date{}

\begin{abstract}
In~\cite{cgens}  and\ccite{cg2}  classes   of initial data to the three dimensional, incompressible Navier-Stokes
 equations were presented, generating a global smooth solution although the norm of 
the initial data may be chosen arbitrarily   large. The aim of this article is to provide new examples of arbitrarily 
large initial data giving rise
 to global solutions, in the whole space. Contrary to the previous examples, the initial data has no particular oscillatory properties, 
 but varies slowly in one direction.  The proof uses the special structure of the nonlinear
 term of the equation.
\end{abstract}


\maketitle

\setcounter{equation}{0}
\section{Introduction}
The purpose of this paper  is to use the special structure of the tridimensional Navier-Stokes equations to 
prove  the global existence of smooth solutions for a class of (large) initial data which are slowly varying  in one direction.
Before entering further in the details, let us recall briefly some classical facts on the global wellposedness of the incompressible 
Navier-Stokes equations  in the whole space~$\R^3$. The equation  itself writes
$$
(NS) \left\{
\begin{array}{c}
\partial_t u +u\cdot \nabla u -\Delta u  =-\nabla p\\
\dive u = 0\\
u_{|t=0} = u_0
\end{array}
\right.
$$ 
where~$u=(u^1,u^2,u^3)= (u^h,u^3)$ is a time dependent vector field on~$\R^3$. The divergence free condition determines $p$ through the relation
$$
 -\Delta p = \sum_{1\leq j,k\leq 3} \partial_j\partial_k(u^ju^k).
$$
This relation allows to put the system~$(NS)$ under the more general form 
$$
(GNS) \left\{
\begin{array}{c}
\partial_t u  -\Delta u =Q(u,u) \\
u_{|t=0} = u_0
\end{array}
\right.
$$
where $\displaystyle Q(v,w)\eqdefa \sum_{1\leq j,k\leq 3} Q_{j,k} (D) (v^jw^k)$ and~$Q_{j,k}(D)$ are smooth
 homogeneous Fourier multipliers of order~$1$.

Moreover, this system has the following scaling invariance: if~$(u,p)$ is a solution on the time
 interval~$[0,T)$, then~$(u_\lam,p_\lam)$ defined by
$$
u_\lam(t,x)\eqdefa \lam u(\lam^2t,\lam x) \and\quad p_\lam (t,x)\eqdefa \lam^2p(\lam^2t,\lam x)
$$
is a solution on the time interval~$[0,\lam^{-2}T)$.
  Of course, any smallness condition on the initial
data that ensures global solutions, must be invariant under the above scaling transformation. The search 
of the ``best" smallness condition is a long story initiated in the seminal paper of J. Leray (see\ccite{leray}), 
continuated in particular by H. Fujita and T. Kato in\ccite{fujitakato}, Y.   Giga  and T. Miyakawa in~\cite{gigamiyakawa},  and M. Cannone, Y. Meyer and F. Planchon
 in\ccite{cannonemeyerplanchon}. This leads to the following theorem proved by H. Koch and D. Tataru in\ccite{kochtataru}. In the
 statement of the theorem, $P(x,R)$ stands for the parabolic set~$[0,R^{2}] \times B(x,R)$
 where~Ê$B(x,R)$ is the ball centered at~$x$, of radius~$R.$
 \begin{theo}[\cite{kochtataru}]
\label{Kochtatarutheo}
{\sl If the initial data~$u_0$ is such that 
\begin{equation}\label{defbmo}
\|u_0\|^2_{BMO^{-1}}\eqdefa \sup_{t>0} t \| e^{t\Delta}u_0\|^2_{L^\infty}
+\supetage {x\in \R^3}{R>0} \frac 1 {R^3} \int_{P(x,R)} |(e^{t\Delta} u_0)(t,y)|^2dy
\end{equation}
is small enough, then there exists a global smooth solution to~$(GNS)$.
}
\end{theo}
A typical example of application  of this theorem is the initial data
\beq
\label{examplekt}
u_{0}^{\e}(x)\eqdefa \cos \Bigl( \frac {x_{3}}\e\Bigr) (\partial_{2}\f
(x_{1},x_{2}),-\partial_{1}\f(x_{1},x_{2}),0) \with \f \in {\mathcal S}(\R^2),
\eeq
as soon as~$\e$ is small enough (see for example\ccite{cg2} for a proof).
The above theorem is probably the end point for the following reason, as observed for instance in\ccite{cg2}.
 If~$B$ is  a Banach space
continuously included in the space~$\cS'$ of tempered distributions on~$\R^3$, such  
that, for any~$(\lam,  a)\in \R^+_{\star}\times \R^3$, Ê$\|f(\lam(\cdot-  a))\|_{B} 
 = \lam^{-1}\|f\|_{B}$, then~$\|\cdot\|_B\leq C \ds \sup_{t>0}  t^{\frac 1 2} \| e^{t\Delta}u_0\|_{L^\infty}$.  
The second 
 condition entering in the definition of the~$BMO^{-1}$ norm given 
in~(\ref{defbmo}) merely translates the fact that the first Picard iterate should be locally square integrable in space and time.

Those results of global existence under a smallness condition do not use the special structure of the incompressible 
Navier-Stokes system and are valid for the larger class of systems of the type~$(GNS)$. The purpose of this paper is to provide   
a class of examples of large initial data which give rise to global smooth solutions for the system~$(NS)$ itself,  and 
not
for the larger class~$(GNS)$. In all that follows, an initial data $u_0$ 
will be said ``large"  if 
\begin{equation}\label{defB-1infty}
\|u_0\|_{\dot B^{-1}_{\infty,\infty}} \eqdefa \sup_{t>0}  t^{\frac 1 2} \|e^{t\Delta}u_0\|_{L^\infty}
\end{equation}
is not small.

Such initial data, in the spirit of the example provided by\refeq{examplekt}, are exhibited in\ccite{cg2}
 (see also\ccite{cgens} for the periodic case). In particular, the following theorem is proved in\ccite{cg2}.
\begin{theo}[\cite{cg2}]
\label{example}
{\sl Let~ $\phi \in {\mathcal S}(\R^3)$ be a given function, and consider two real
numbers~$\varepsilon$ and~$\alpha$
in~$]0,1[$. Define
$$
\varphi^{\varepsilon}(x) = \frac{({-\log \varepsilon})^{\frac15}}
{\e^{1-\alpha}} \cos  \left(
\frac{x_{3}}{\e}\right)
\phi\Bigl(x_{1}, \frac{x_{2}}{\varepsilon^{\alpha}}, x_{3}\Bigr).
$$
Then  for~$\varepsilon$ small enough, the smooth,
divergence free vector field
$$
u_{0}^{\e}(x) = (\partial_{2}\varphi^{\varepsilon}(x),
-\partial_{1}\varphi^{\varepsilon}(x), 0)
$$
satisfies~$\displaystyle \lim_{\e \rightarrow 0}  \|u_{0}^{\varepsilon}\|_{\dot B^{-1}_{\infty,\infty}}
= \infty$,  
and generates a unique global solution to~$(NS)$.
}\end{theo}

The proof of that theorem shows that the solution remains close to  the solution of the free equation~$e^{t\Delta}u_{0,\e}$. 
It is  important to notice   that the proof uses in a crucial way the algebraic structure of the non linear term~$u\cdot\nabla u$, 
but uses neither  the energy estimate, nor the fact that the two dimensional,
 incompressible Navier-Stokes is globally wellposed.  Let us give some other results on large initial data giving rise
to global solutions: In~\cite{babinandco},  the initial
data is chosen so as to transform the equation into a rotating fluid equation.  In~\cite{iftimieraugelsell}, 
\cite{raugelsell}  the equations are posed in a thin domain (in all those cases
the global wellposedness  of the two dimensional equation is a crucial ingredient in the proof). 
In~\cite{gimrn} and~\cite{iftimie}, the case
of an initial data close to bidimensional vector field is studied, in the periodic case.  Finally   in~\cite{foiassaut2} 
(Remark~7),    an arbitrary
large initial data is constructed in the periodic case, generating a global solution (which is in fact a solution to the heat equation, as the 
special dependence on the space variables implies that the nonlinear term cancels).

The class of examples we exhibit here is quite different. They are close to a  two   dimensional flow in the sense that they are slowly 
varying in the one direction (the vertical one). More precisely the aim of this paper is the proof of the following theorem.
\begin{theo}
\label{theoquasi2D}
{\sl Let~$v_0^h=(v_0^1,v_0^2,0)$ be a  horizontal, smooth
 divergence free  vector field on~$\R^3$ (i.e.~$v_0^h$ is  in~$L^2(\R^3)$ as well as all its derivatives), 
 belonging, as well as all its derivatives, to~$L^2(\R_{x_3}; \dot H^{-1}(\R^2))$; let $w_0$ be 
 a smooth divergence free vector field on~$\R^3$. Then if $\e$ is small enough, the initial data
$$
u_{0}^{\e} (x) = ( v_0^h+\e w^h,w_0^3 )(x_h,\e x_3)
$$
generates a unique, global   solution~$u^{\e}$ of~$(NS)$.
}
\end{theo}

\noindent{\bf Remarks}
\begin{itemize}
\item A typical example of vector fields~$v_0^h$ satisfying the hypothesis
 is~$v_0^h= (-\partial_2 \f,\partial_1\f, 0)$ where~$\f$ is a function of the Schwarz class~$\cS(\R^3)$.
\item This class of examples of initial data corresponds to a ``well prepared" case.  The ``ill prepared" case would correspond
 to the case 
when the horizontal divergence of the initial data is of size~$\e^\alpha$ with $\alpha$ less than  $1$,  and   the vertical 
component  of 
 the initial data is of size~$\e^{\alpha-1}$. This case  is certainly   very interesting to understand,  but that goes probably far beyond
 the methods
used in this paper.
\item We have to check that the initial data may be large. This is ensured by the following proposition.
\end{itemize}
\begin{prop}
\label{quasi2Dlargecheck}
{\sl Let $(f,g)$ be in~$  \cS(\R^3)$. Let us define~$h^\e (x_h,x_3) \eqdefa f(x_h)g(\e x_3)$. We have, if $\e$ is small enough,
$$
\|h^\e\|_{\dot B^{-1}_{\infty,\infty}(\R^3)} \geq \frac 1 4 \|f\|_{\dot B^{-1}_{\infty,\infty}(\R^2)} 
\|g\|_{L^\infty(\R)}.
$$
}
\end{prop}
\begin{proof}
By the definition of~$\|\cdot\|_{\dot B^{-1}_{\infty,\infty}(\R^3)}$ given by~(\ref{defB-1infty}), 
 we have to bound from below the quantity~$\|e^{t\Delta}h^\e
\|_{L^\infty(\R^3)}$.  Let us write that
$$
(e^{t\Delta} h^\e) (t,x)  =  (e^{t\Delta_h} f)(t,x_h)  (e^{t\partial_3^2} g)(\e^2t,\e x_3).
$$
Let us consider a positive time~$t_0$ such that
$$
t_0^{\frac 1 2} \|e^{t_0\Delta_h} f\|_{L^\infty(\R^2)} \geq  \frac 1 2\|f\|_{\dot B^{-1}_{\infty,\infty}(\R^2)} .
$$
Then we have
\beno
t_0^{\frac 1 2} \|e^{t_0\Delta } h^\e\|_{L^\infty(\R^3)} & = & t_0^{\frac 1 2} \|e^{t_0\Delta_h} f
\|_{L^\infty(\R^2)} \|(e^{t_0\partial_3^2} g)(\e^2t_0,\e \cdot)\|_{L^\infty(\R)}\\
 & \geq & \frac 1 2\|f\|_{\dot B^{-1}_{\infty,\infty}(\R^2)} \|e^{\e^2t_0\partial_3^2} g\|_{L^\infty(\R)}.
\eeno
As~$\ds \lim_{\e\rightarrow0} e^{\e^2t_0\partial_3^2} g=g$ in~$L^\infty(\R)$, the proposition is proved.
\end{proof}

{\bf Structure of the paper:} The proof of   Theorem~\ref{theoquasi2D} is achieved in the next section, assuming two crucial 
lemmas. The proof of those lemmas is postponed to Sections~\ref{estimatesvapp} and~\ref{estimateserror} respectively.

{\bf Notation:}    If~$A$ and~$B$ are two real numbers, we shall write~$A  \lesssim B$ if there is
a universal constant~$C$, which does not depend on varying parameters of the problem, such that~$A \leq CB$. If~$A  \lesssim B$
and~$B  \lesssim A$, then we shall write~$A \sim B$.

If~$v_{0}$ is a vector field, then we shall denote by~$C_{v_{0}}$ a constant
depending only on  norms of~$v_{0}$. Similarly we shall use the notation~$C_{v_{0},w_{0}}$ if the constant depends
on norms of two vector fields~$v_{0}$ and~$w_{0}$, etc.

A function space with a subscript ``$h$''  (for ``horizontal'') will denote a space defined on~$\R^{2}$, while the 
subscript~``$v$''  (for ``vertical'') will denote a space defined on~$\R$. For
 instance~$L^{p}_{h} \eqdefa  L^{p}(\R^{2})$, $L^{q}_{v} \eqdefa L^{q}(\R)$, and similarly
for Sobolev spaces or for mixed spaces such as~$L^{p}_{v} L^{q}_{h} $ or~$L^p_v\dot H^\s_h$.


\section{Proof of the theorem}\setcounter{equation}{0}
\label{sructureproofquasi2D}

The proof of Theorem~\ref{theoquasi2D} consists  in constructing an approximate solution to~$(NS)$ as a perturbation to 
the 2D Navier-Stokes system. Following the idea that we are close to  the two dimensional, periodic
 incompressible Navier-Stokes system, let 
us define~$\vp^h$ as the solution of the following system, where~$y_{3} \in \R$ is a parameter:
$$
(NS2D_{3}) \left\{
\begin{array}{c}
\partial_t \vp^h + \vp^h \cdot \nabla_h \vp^h -\Delta_h  \vp^h = -\nabla_h \presspO \quad \mbox{in} \: 
\R^+ \times \R^2\\
\dive_h \vp^h = 0\\
\vp^h_{|t=0} = v_0^h(\cdot ,y_3).
\end{array}
\right.
$$ 
This system is globally wellposed for any~$y_{3} \in \R$,  and the solution is smooth in (two dimensional) space, and in time. 
Let us consider the solution~$\wp^\e$ 
of the linear equation
$$
(T^\e_{\vp}) \left\{
\begin{array}{c}
\partial_t \wp^\e + \vp^h \cdot \nabla_h \wp^\e -\Delta_h  \wp^\e -\e^2\partial_3^2 \wp^\e =
-( \nabla^h \presspun,\e^2 \partial_{3} \presspun)\quad \mbox{in} \: 
\R^+ \times \R^3\\
\dive \wp^\e = 0\\
\wp^\e_{|t=0} = w_0,
\end{array}
\right.
$$ 
and let us define the approximate solution
\ben
\label{definvapp}
\vapp (t,x) & = &  (( \vp^h, 0 )+\e (\wp^{\e,h},\e^{-1}\wp^{\e,3}
 ))(t,x_h,\e x_3)\and\\
 \nonumber
  p^\e_{app}(t,x) & =  & (\presspO + \e   \presspun) (t,x_h,\e x_3).  
  \een
Finally let us consider the unique smooth solution~$u^\e$ of~$(NS)$ associated with the initial
data~$u_{0}^{\e}$ on its maximal time  interval of existence~$[0,T_\e)$. The 
proof of Theorem~\ref{theoquasi2D} consists in proving   global in time estimates 
on~$\vapp $, in order to prove  that~$R^{\e} \eqdefa u^{\e}
- \vapp$ remains small, globally in time; this   ensures the global regularity for~$(NS)$.

More precisely, the proof of Theorem~\ref{theoquasi2D}  relies on the following two lemmas, whose proofs are postponed
to Sections~\ref{estimatesvapp} and~\ref{estimateserror} respectively. 
\begin{lemma}
\label{lemmaquasi2D1}
{\sl The vector field~$\vapp$ defined in~(\ref{definvapp}) satisfies the following estimate:
$$
\|\vapp\|_{L^2(\R^+;L^\infty(\R^3))} + \|\nabla \vapp\|_{L^2(\R^+;L^\infty_vL^2_h)}\leq
 C_{v_0,w_0}.
$$
}
\end{lemma}

\begin{lemma}
\label{lemmaquasi2D2}
{\sl The vector field~$R^{\e} \eqdefa u^{\e}
- \vapp$ satisfies the equation
$$
(E^\e) \left\{
\begin{array}{c}
\partial_t R^\e + R^\e \cdot \nabla R^\e -\Delta R^\e+\vapp \cdot \nabla R^\e+R^\e \cdot \nabla \vapp =
 F^\e-\nabla q^\e\\
\dive R^\e = 0\\
R^\e_{|t=0} = 0
\end{array}
\right.
$$ 
with $\|F^\e\|_{L^2(\R^+;\dot H^{-\frac 1 2}(\R^3))} \leq C_{v_0,w_0} \e^{\frac 1 3}$.
}
\end{lemma}
 Let us postpone the proof of those lemmas    and conclude the proof of Theorem\refer{theoquasi2D}.
We denote, for any positive~$\lam$,
$$
V_\e(t)\eqdefa \|\vapp(t,\cdot)\|_{L^\infty(\R^3)}^{2}
 + \|\nabla \vapp(t,\cdot)\|_{L^\infty_vL^2_h}^{2}\and 
R^\e_\lam (t)\eqdefa \exp \left(- \lam\int_0^t V_\e(t')dt'\right) R^{\e}(t).
$$
Lemma\refer{definvapp} implies
 that~$\ds I_0\eqdefa \int_0^\infty V_\e(t)dt$ is finite. By an~$\dot H^{\frac 1 2}$ energy estimate in~$\R^3$,
 we get
$$
\longformule{
\frac 1 2\frac d {dt} \|R^\e_\lam(t)\|^2_{\dot H^{\frac 1 2}} 
+ \|\nabla R^\e_\lam(t)\|^2_{\dot H^{\frac 1 2}} 
\leq -2\lam V_\e(t)\|R^\e_\lam(t)\|^2_{\dot H^{\frac 1 2}}
+ e^{\lam I_0} \left| (R^\e_\lam(t) \cdot \nabla
 R^\e_\lam(t)|R_\lam^\e(t))_{\dot H^{\frac 1 2}} \right|
}
{
{}+\left| (R^\e_\lam(t) \cdot \nabla \vapp(t)|R_\lam^\e(t))_{\dot H^{\frac 1 2}}\right|
+\left| (\vapp(t)\cdot R^\e_\lam(t)|R_\lam^\e(t))_{\dot H^{\frac 1 2}} \right|
+\left| (F^\e(t)|R_\lam^\e(t))_{\dot H^{\frac 1 2}} \right| .
}
$$
The estimate (i) of Lemma 1.1 of\ccite{chemin10} claims  in particular that, for any~$s\in ]-d/2,d/2[$,
 for any divergence free vector field~$a$ in~$d$ space dimensions and any function $b$, we have
\beq
\label{estimesiam}
(a\cdot\nabla b|b)_{\dot H^s } \leq C \|\nabla a\|_{\dot H^{\frac d 2-1} } \| b\|_{\dot H^s } 
\|\nabla b\|_{\dot H^s}.
\eeq
Applying with~$d=3$ and~$s=1/2$, this implies that
\beq 
\label{estimconclutheo31}
 \left| (R^\e_\lam(t) \cdot \nabla R^\e_\lam(t)|R_\lam^\e(t))_{\dot H^{\frac 1 2}} \right|
\lesssim \|R_\lam^\e(t)\|_{\dot H^{\frac 1 2}}\|\nabla  R_\lam^\e(t)\|^2_{\dot H^{\frac 1 2}}.
\eeq
In order to estimate the other non linear terms, let us establish the following lemma.
\begin{lemma}
\label{lemmeconclutheo31}
{\sl Let~$a$ and~$b$ be two vector fields. We have
$$
\left|(a\cdot\nabla b|b)_{\dot H^{\frac 1 2}} \right|+
\left|(b\cdot\nabla a|b)_{\dot H^{\frac 1 2}} \right| 
\lesssim \left(\|a\|_{L^\infty} +\|\nabla a\|_{L^\infty_v(L^2_h)}\right)
\|b\|_{\dot H^{\frac 1 2}}\|\nabla b\|_{\dot H^{\frac 1 2}}.
$$
}
\end{lemma}
\begin{proof}
By definition of the~$\dot H^{\frac 1 2}$ scalar product, we have
\beno
(a\cdot\nabla b|b)_{\dot H^{\frac 1 2}} &  \leq & \|a\cdot\nabla b\|_{L^2}\|\nabla b\|_{L^2}\\
& \leq & \|a\|_{L^\infty} \|\nabla b\|^2_{L^2}.
\eeno
The interpolation inequality between Sobolev norm gives 
$$
(a\cdot\nabla b|b)_{\dot H^{\frac 1 2}}   \leq \|a\|_{L^\infty} \|b\|_{\dot H^{\frac 1 2}} \|\nabla b\|_{\dot H^{\frac 1 2}}.
$$
Now let us estimate~$(b\cdot\nabla a|b)_{\dot H^{\frac 1 2}}$. Again we use that 
$$
(b\cdot\nabla a|b)_{\dot H^{\frac 1 2}} \leq \|b\cdot\nabla a\|_{L^2}\|\nabla b\|_{L^2}.
$$
Then let us write that
$$
 \|b\cdot\nabla a\|_{L^2}^2 = \int_{\R^3} |b(x_h,x_3)\nabla a(x_h,x_3)|^2dx_hdx_3.
$$
Gagliardo-Nirenberg's inequality in the horizontal variable implies that
$$
\forall x_{3} \in \R, \quad |b(x_h, x_3)|^2 \lesssim \|b(\cdot,x_3)\|_{\dot H_h^{\frac 1 2}}
\|\nabla_h b(\cdot,x_3)\|_{\dot H_h^{\frac 1 2}}.
$$
Let us use the Cauchy-Schwarz inequality; this gives
\beno
 \|b\cdot\nabla a\|_{L^2}^2  & \leq &  \int_{\R}  \|b(\cdot ,x_3)\|_{\dot H_h^{\frac 1 2}}
 \|\nabla_h b(\cdot,x_3)\|_{\dot H_h^{\frac 1 2}} \|\nabla a(\cdot,x_3)\|^{2}_{L^2_h} dx_3\\
  & \leq  & \|\nabla a\|^{2}_{L^\infty_v(L^2_h)}  \int_{\R}  \|b(\cdot ,x_3)\|_{\dot H_h^{\frac 1 2}}
 \|\nabla_h b(\cdot,x_3)\|_{\dot H_h^{\frac 1 2}}  dx_3 \\
  & \leq & \|\nabla a\|^{2}_{L^\infty_vL^2_h}  \|b\|_{L^2_v\dot H_h^{\frac 1 2}}
 \|\nabla_h b\|_{L^2_v\dot H_h^{\frac 1 2}}.
\eeno
When $s$  is positive,  we have, thanks to Fourier-Plancherel in the vertical variable,
\beno
\|b\|_{L^2_v(\dot H_h^s)} & = &  \int_{\R} \int_{\R^{2}}
 |\xi_h|^{2s} |\cF_hb(\xi_h,x_3)|^2 d\xi_hdx_3\\
& \sim &   \int_{\R} \int_{\R^{2}}  |\xi_h|^{2s} |\wh b(\xi_h,\xi_3)|^2d\xi_hd\xi_{3}\\
& \lesssim  &   \int_{\R^3}    |\xi|^{2s} |\wh b(\xi)|^2d\xi\\
& \lesssim &   \|b\|_{\dot H^s}^2.
\eeno
This concludes the proof of Lemma~\ref{lemmeconclutheo31}.
\end{proof}
\noindent{\bf Conclusion of the proof of Theorem\refer{theoquasi2D}. } We infer from the above lemma that
$$
\longformule
{
\left| (R^\e_\lam(t) \cdot \nabla \vapp(t)|R_\lam^\e(t))_{\dot H^{\frac 1 2}}\right|
+\left| (\vapp(t)\cdot R^\e_\lam(t)|R_\lam^\e(t))_{\dot H^{\frac 1 2}} \right| 
}
{ {}
\leq
\frac 1 4 \|\nabla R^\e_\lam(t)\|_{\dot H^{\frac 1 2}}^2
+CV_\e(t)\|R^\e_\lam(t)\|^2_{\dot H^{\frac 1 2}}.
}
$$
Together with\refeq{estimconclutheo31}, this gives
$$
\longformule{
\frac 1 2\frac d {dt} \|R^\e_\lam(t)\|^2_{\dot H^{\frac 1 2}} 
+ \|\nabla R^\e_\lam(t)\|^2_{\dot H^{\frac 1 2}} 
\leq (C-2\lam) V_\e(t)\|R^\e_\lam(t)\|^2_{\dot H^{\frac 1 2}}
}
{
{}+C e^{\lam I_0} \|R^\e_\lam(t)\|_{\dot H^{\frac 1 2}} \|\nabla R^\e_\lam(t)\|^2_{\dot H^{\frac 1 2}} 
+C\|F_\e(t)\|^2_{\dot H^{\frac 1 2}}.
}
$$
Choosing~$\lam$ such that $C-2\lam$ is negative, we infer that
$$
\frac d {dt} \|R^\e_\lam(t)\|^2_{\dot H^{\frac 1 2}} 
+(1-Ce^{\lam I_0} ) \|\nabla R^\e_\lam(t)\|^2_{\dot H^{\frac 1 2}} 
\leq C\|F_\e(t)\|^2_{\dot H^{\frac 1 2}}.
$$
Since~$R^\e(0)=0$, we get, as long as~$\|R^\e_\lam(t)\|_{\dot H^{\frac 1 2}}$ is less or equal to~$1/2 
Ce^{-\lam I_0}$, that
$$
 \|R^\e_\lam(t)\|^2_{\dot H^{\frac 1 2}} +\frac 1 2 \int_0^t  \|\nabla 
R^\e_\lam(t')\|^2_{\dot H^{\frac 1 2}} dt' \leq C_{v_0,w_0} \e ^{\frac 1 3}.
$$
We therefore obtain that~$R^\e$ goes to zero in~$L^{\infty}(\R^+;\dot H^{\frac 1 2}) \cap 
L^{2}(\R^+;\dot H^{\frac 3 2}) $. That implies that~$u^{\e}$ remains close for all times to~$\vapp$, which in particular
implies  Theorem\refer{theoquasi2D}. \hfill $\Box$


\section{Estimates on the approximate solution}\label{estimatesvapp}\setcounter{equation}{0}
In this section we shall prove Lemma~\ref{lemmaquasi2D1} stated in the previous section. The proof of the lemma
is achieved by obtaining estimates on~$\vp$, stated in the next lemma, as well as on~$\wpe$ (see
 Lemma~\ref{lemmaestimateprofil2} below).  
\begin{lemma}
\label{lemmaestimateprofil1}
{\sl Let $\vp^h$ be a solution of the system~$(NS2D_3)$. Then, for any $s$ greater than~$-1$ and any $\alpha \in \N^3$, we have, for any $y_3$ in~$\R$ and for any positive~$t$,
$$
\|\partial^\al \vp^h(t,\cdot,y_3)\|_{\dot H^s_h} ^2+\int_0^t
\|\partial^\al \nabla_h\vp^h(t',\cdot,y_3)\|_{\dot H^s_h} ^2 dt' \leq C_{v_0}(y_3),
$$
where~$C_{v_0}(\cdot)$ belongs to~$L^{1} \cap L^{\infty}(\R)$ and its norm is controled by a constant~$C_{v_{0}}$.
}
\end{lemma}

\begin{proof}
For~$s=0$ and~$\alpha =0$, the estimate is simply the classical~$L^2$-energy estimate with~$y_3$ as a parameter:
writing~$\vp = (\vp^h,0)$ we have
\beq
\label{estimenergy2Dparameter}
\|\vp(t,\cdot,y_3)\|_{L^2_h} ^2+2\int_0^t
\| \nabla_h\vp(t',\cdot,y_3)\|_{L^2_h} ^2 dt' =\|\vp_0(\cdot,y_3)\|_{L^2_h} ^2.
\eeq
In the case when~$\alpha=0$,  the estimate (i) in Lemma 1.1 of\ccite{chemin10} gives, for any $s$ greater than~$-1$,
$$
\frac 1 2 \frac d {dt}  \|\vp (t,\cdot,y_3)\|^2_{\dot H^s_h} + 
\|\nabla_h\vp (t,\cdot,y_3)\|^2_{\dot H^s_h} \leq C \| \nabla_h\vp(t,\cdot,y_3)\|_{L^2_h}
 \|\vp (t,\cdot,y_3)\|_{\dot H^s_h} \|\nabla_h\vp (t,\cdot,y_3)\|_{\dot H^s_h}.
$$
We infer that
$$
 \frac d {dt}  \|\vp (t,\cdot,y_3)\|^2_{\dot H^s_h} + \|\nabla_h\vp (t,\cdot,y_3)\|^2_{\dot H^s_h} \leq C \| \nabla_h\vp(t,\cdot,y_3)\|^2_{L^2_h} \|\vp (t,\cdot,y_3)\|_{\dot H^s_h} ^2.
$$
Gronwall's lemma ensures that
$$
\|\vp (t,\cdot,y_3)\|^2_{\dot H^s_h} + \int_0^t\|\nabla_h\vp (t',\cdot,y_3)\|^2_{\dot H^s_h} dt'
 \leq \|\vp_0 (\cdot,y_3)\|_{\dot H^s_h} ^2\exp \Bigl (C\!\int_0^t \| \nabla_h\vp(t',\cdot,y_3)\|^2_{L^2_h} dt'\Bigr).
$$
The energy estimate\refeq{estimenergy2Dparameter} implies that
$$
\|\vp (t,\cdot,y_3)\|^2_{\dot H^s_h} + \int_0^t\|\nabla_h\vp (t',\cdot,y_3)\|^2_{\dot H^s_h} dt'
 \leq \|\vp_0 (\cdot,y_3)\|_{\dot H^s_h} ^2\exp \left (C \|\vp_0\|_{L^\infty_vL^2_h} ^2\right).
$$
This proves the lemma in the case when~$\al=0$.  Let us now turn to the general case, by induction on the
 length of~$\alpha$.
It is clear that in the proof, we can restrict
 ourselves to the case when $s\in ]-1,1[$. 

Let us assume that, for some~$k \in \N$, 
\beq
\label{eq1demolemmaestimateprofil1}
\forall s \in ]-1,1[, \quad \sum_{|\al|\leq k} \biggl(\|\partial^\al \vp(t,\cdot,y_3)\|_{\dot H^s_h} ^2+\int_0^t
\|\partial^\al \nabla_h\vp(t',\cdot,y_3)\|_{\dot H^s_h} ^2 dt'\biggr)
\leq  C_{k,v_0}(y_3),
\eeq
with~$C_{k,v_0}(\cdot) \in L^{1} \cap L^{\infty}(\R)$.

Thanks to the Leibnitz formula  we have, for~$|\alpha| \leq k+1$,
$$
\partial_t \partial^\al \vp^h + \vp^h \cdot \nabla_h  \partial^\al  \vp^h -\Delta_h  \partial^\al  \vp^h = -\nabla_h p_\alpha -\sumetage{\beta\leq \alpha}{\beta\not =\alpha} C_\alpha^\beta  \partial^{\al -\beta} \vp^h\cdot\nabla_h\ \partial^\beta\vp^h.
$$
Performing a~$\dot H^s_h$ energy estimate in the horizontal variable and using the 
estimate\refeq{estimesiam} in the case when~$d=2$ gives 
$$
\displaylines{
\frac 1 2 \frac d {dt}  \| \partial^\al\vp (t,\cdot,y_3)\|^2_{\dot H^s_h} + \|\nabla_h\partial^\al\vp (t,\cdot,y_3)\|^2_{\dot H^s_h} \cr\leq C \| \nabla_h\vp(t,\cdot,y_3)\|_{L^2_h} \| \partial^\al\vp (t,\cdot,y_3)\|_{\dot H^s_h} \|\nabla_h \partial^\al\vp (t,\cdot,y_3)\|_{\dot H^s_h}\cr
+C_\al \sumetage{\beta\leq \alpha}{\beta\not =\alpha} \Bigl|
\left(\partial^{\al -\beta} \vp^h(t,\cdot,y_3)\cdot\nabla_h\ \partial^\beta\vp^h(t,\cdot,y_3)
\big|\partial^\al\vp¬h (t,\cdot,y_3)\right)_{\dot H^s_h}\Bigr|.
}
$$
To estimate the last term, we shall treat differently   the case~$|\beta |= 0$ and~$|\beta |\neq 0$. In the first case, 
we
notice first that when~$s = 0$,  laws of product for Sobolev spaces in~$\R^2$ give
\begin{eqnarray*}
 \! \! \! \!  \! \! \! \! \!  \!    \left(\partial^{\al } \vp^h(t,\cdot,y_3)\cdot\nabla_h \vp^h(t,\cdot,y_3)
\big|\partial^\al\vp^h (t,\cdot,y_3)\right)_{L^{2}_h} & \lesssim & \|\partial^{\al } 
\vp^h(t,\cdot,y_3)\|_{\dot H^{\frac12}}^{2} \| \nabla_h \vp^h(t,\cdot,y_3) \|_{L^{2}}\\
  \lesssim \|\partial^{\al } 
\vp^h(t,\cdot,y_3)\|_{L^{2}_{h} } &&   \! \! \! \!  \! \! \! \! \!  \! \! \! \! \! \! 
\| \nabla_h \partial^{\al } 
\vp^h(t,\cdot,y_3) \|_{L^{2}_{h}} \| \nabla_h \vp^h(t,\cdot,y_3) \|_{L^{2}_{h}}.
\end{eqnarray*}
 If~$s>0$,  
then again  laws of product for Sobolev spaces in~$\R^2$ give, for~$s\in ]0,1[$,
$$
\Bigl| \left(\partial^{\al } \vp^h(t,\cdot,y_3)\cdot\nabla_h \vp^h(t,\cdot,y_3)
\big|\partial^\al\vp (t,\cdot,y_3)\right)_{\dot 
H^{s}_h}\Bigr| \lesssim \|\partial^\al\vp^{h}\|_{\dot H^{s}_h}
\|\nabla_h \vp^h\|_{L^{2}_{h}} \|\nabla_h\partial^\al\vp^{h}\|_{\dot H^{s}_h},
$$
whereas if~$-1<s<0$,  
$$
\Bigl| \left(\partial^{\al } \vp^h(t,\cdot,y_3)\cdot\nabla_h \vp^h(t,\cdot,y_3)
\big|\partial^\al\vp (t,\cdot,y_3)\right)_{\dot H^{s}_h}\Bigr| \lesssim \|\nabla_h
\partial^\al\vp^{h}\|_{\dot H^{s}_h}
\|\nabla_h \vp^h\|_{L^{2}_{h}} \|\partial^\al\vp^{h}\|_{\dot H^{s}_h}.
$$
So in any case we have
$$
\Bigl| \left(\partial^{\al } \vp^h(t,\cdot,y_3)\cdot\nabla_h \vp^h(t,\cdot,y_3)
\big|\partial^\al\vp (t,\cdot,y_3)\right)_{\dot H^{s}_h}\Bigr| \leq
\frac 14  \|\nabla_h\partial^\al\vp^{h}\|_{\dot H^{s}_h}^{2} + C
\|\nabla_h \vp^h\|_{L^{2}_{h}}^{2} \|\partial^\al\vp^{h}\|_{\dot H^{s}_h}^{2}.
$$
Now let us consider the case when~$|\beta| \neq 0$. As the horizontal divergence of~$\vp$ is identically~$0$, we have
$$
\Bigl|
\left(\partial^{\al -\beta} \vp^h\cdot\nabla_h\ \partial^\beta\vp^h\big|\partial^\al\vp (t,\cdot,y_3)\right)_{\dot H^s_h}\Bigr| \leq
\|\partial^{\al -\beta} \vp^h(t,\cdot,y_3)\otimes \partial^\beta\vp^h(t,\cdot,y_3)\|_{\dot H^s_h} 
\|\nabla_h\partial^\al\vp (t,\cdot,y_3)\|_{\dot H^s_h}.
$$
Laws of product for Sobolev spaces in~$\R^2$ give, for~$s\in ]-1,1[$,  
$$ 
\|\partial^{\al -\beta} \vp^h\otimes \partial^\beta\vp^h\|_{\dot H^s_h} \leq C
\|\partial^{\al -\beta} \vp^h\|_{\dot H^{s'}_h}\|\partial^{\beta}\nabla_h\vp^h\|_{\dot H^{s-s'}_h},
$$
where~$s'$ is chosen so that~$s<s'<1$.

Finally we deduce that
$$
\longformule{
  \frac d {dt} 
 \| \partial^\al\vp (t,\cdot,y_3)\|^2_{\dot H^s_h} + \|\nabla_h\partial^\al\vp (t,\cdot,y_3)\|
^2_{\dot H^s_h}
\leq C \|\nabla_h \vp \|_{L^{2}_{h}}^{2} \|\partial^\al\vp \|_{\dot H^{s}_h}^{2}
}
{{} +
C_\al 
 \sumetage{\beta \leq \alpha}{\beta\notin \{0,\alpha\}}
\|\partial^{\al -\beta} \vp  (t,\cdot,y_3)\|_{\dot H^{s'}_h}
\|\partial^{\beta}\nabla_h\vp  (t,\cdot,y_3)\|_{\dot H^{s-s'}_h} 
\|\nabla_h \partial^\al\vp (t,\cdot,y_3)\|_{\dot H^s_h}.
}
$$
Gronwall's lemma together with the induction hypothesis\refeq{eq1demolemmaestimateprofil1}
implies that
$$
\longformule{
 \sum_{|\al|= k+1} \biggl(\|\partial^\al \vp(t,\cdot,y_3)\|_{\dot H^s_h} ^2+\int_0^t
\|\partial^\al \nabla_h\vp(t',\cdot,y_3)\|_{\dot H^s_h} ^2 dt'\biggr)
}
{{} \lesssim \biggl( \sum_{|\al|= k+1}\|\partial^\al \vp_0(\cdot,y_3)\|_{\dot H^s_h} ^2+ C_{k,v_0}(y_3)\biggr)
 \exp \Bigl (C_k\int_0^t \| \nabla_h\vp(t',\cdot,y_3)\|^2_{L^2_h} dt'\Bigr).
}
$$
The~$L^2$ energy estimate\refeq{estimenergy2Dparameter} allows to conclude the proof of 
Lemma~\ref{lemmaestimateprofil1}.
\end{proof}
From this lemma, we deduce the following corollary.
\begin{corol}
\label{corestimateprofil1}
{\sl Let $\vp^h$ be a solution of the system~$(NS2D_3)$. Then, for any non negative~$\sigma$, we have
$$
\|\vp^h\|_{L^2(\R^+;\dot H^\sigma(\R^3))} \leq C_{v_0}\and \|\partial ^\alpha \vp^h\|_{L^2(\R^+;L^\infty_v\dot H^\s_h)} \leq C_{v_0}.
$$
}
\end{corol}
\begin{proof}
To start with, let us assume~$\s>0$. Lemma\refer{lemmaestimateprofil1}  applied with~$ s=\s-1$ implies that
\beq
\label{corestimateprofil1eq1}
\forall \s >0\,,\ \forall \al \in \N\,,\  \|\partial ^\alpha \vp^h\|_{L^2(\R^+;L^2_v\dot H^\s_h)}
 \leq C_{v_0}.
\eeq
Then, for any non negative~$\s$, we have
\beno
\|\partial ^\al \vp (t,\cdot,y_3)\|^2_{\dot H^\s_h} & = & 2\int_{-\infty} ^{y_3} (\partial_3\partial^\al \vp (t,\cdot ,y'_3)|\partial^\al \vp (t,\cdot ,y'_3)_{\dot H^\s_h} dy'_3\\
 & \leq & 2  \|\partial_3\partial ^\alpha \vp^h(t,\cdot)\|_{L^2_v\dot H^\s_h}\|\partial ^\alpha \vp^h(t,\cdot)\|_{L^2_v\dot H^\s_h}
\eeno
By the Cauchy-Schwarz inequality, we have
$$
\forall \s \geq 0\,,\ \forall \al \in \N\,,\  \|\partial ^\alpha \vp^h\|_{L^2(\R^+;L^\infty_v\dot H^\s_h)} \leq \|\partial_3\partial ^\alpha \vp^h\|_{L^2(\R^+;L^\infty_v\dot H^\s_h)}^{\frac 1 2}\|\partial ^\alpha \vp^h\|_{L^2(\R^+;L^\infty_v\dot H^\s_h)}^{\frac 1 2}.
$$
From\refeq{corestimateprofil1eq1}, we infer
\beq
\label{corestimateprofil1eq3}
\forall \s >0\,,\ \forall \al \in \N\,,\  \|\partial ^\alpha \vp^h\|_{L^2(\R^+;L^\infty_v\dot H^\s_h)} \leq C_{v_0}.
\eeq
Now, by interpolation,  it is enough to prove the first inequality with~$\s=0$. The system~$(NS2D_3)$ can be written
$$
\left\{\begin{array}{c}
\partial_t \vp -\Delta_h \vp =f\\
\vp_{|t=0 }= \vp_0(\cdot,y_3)
\end{array}
\right. \with f\eqdefa \sum_{1\leq j,k\leq 2} Q_{j,k}(D) (\vp^j\vp^k).
$$
where~$Q_{j,k}$ are homogenenous smooth Fourier multipliers of order~$1$. By Sobolev embeddings in~$\R^2$, we get, for any~$y_3$ in~Ê$\R$,
\beno
\|\vp(\cdot,y_3)\|_{L^2(\R^+\times \R^2)} & \leq & \|\vp_0(\cdot,y_3)\|_{\dot H_h^{-1}} +\|f(\cdot,y_3)\|_{L^1(\R^+;\dot H_h^{-1})}\\
& \leq &  \|\vp_0(\cdot,y_3)\|_{\dot H_h^{-1}} +C \|\vp (\cdot,y_3)\|_{L^2(\R^+;\dot H^{\frac 1 2}_h)}^2\\
 & \leq  &  \|\vp_0(\cdot,y_3)\|_{\dot H_h^{-1}} +C \|\vp (\cdot,y_3)\|_{L^2(\R^+;\dot H^{\frac 1 2}_h)}\sup_{y_3}\|\vp (\cdot,y_3)\|_{L^2(\R^+;\dot H^{\frac 1 2}_h)}.
\eeno
As~$\sup_{y_3}\|\vp (\cdot,y_3)\|_{L^2(\R^+;\dot H^{\frac 1 2}_h)}\leq \|\vp\|_{L^2(\R^+;L^\infty_v\dot H^{\frac 1 2}_h)}$, we infer from\refeq{corestimateprofil1eq3}
$$
\|\vp\|_{L^2(\R^+\times \R^3)}^2 \lesssim  \|\vp_0\|_{L^2_v\dot H^{-1}_h}^2+C_{v_0}
\|\vp\|^2_{L^2(\R^+;L^2_v\dot H^{\frac 1 2}_h)} \leq C_{v_0}.
$$
The corollary is proved.
\end{proof}
Finally we have the following estimate on~$\wpe$.
\begin{lemma}
\label{lemmaestimateprofil2}
{\sl Let $\wpe$ be a solution of the system~$(T^\e_\vp)$. Then, for any $s$ greater than~$-1$ 
and any $\alpha \in \N^3$ and for any positive~$t$, we have
$$
\|\partial^\al \wpe(t,\cdot)\|_{L^2_v\dot H^s_h} ^2+\int_0^t
\|\partial^\al \nabla_h\wpe(t',\cdot)\|_{L^2_v\dot H^s_h} ^2 dt' \leq C_{v_0,w_0}.
$$
}
\end{lemma}
\begin{proof}
We shall only sketch the proof, as it is very close to the proof of Lemma~\ref{lemmaestimateprofil1} which was 
carried out above. The only difference is that the horizontal divergence of~$\wp$ does not vanish identically, but that
will not change very much the estimates. We shall only write the proof in the case when~$\alpha = 0$ and~$-1<s<1$, and
leave the general case to the reader. Using Lemma~1.1 of~\cite{chemin10} we have, for any~$y_3$ in~$\R$,
$$
\left(\vp^h(t,\cdot,y_3)\cdot\nabla_h \wpe(t,\cdot,y_3)|\wpe(t,\cdot,y_3)\right)_{\dot H^s_h} \leq C
\|\nabla_h\vp (t,\cdot,y_3)\|_{L^2_h}\|\nabla_h\wpe (t,\cdot,y_3)\|_{\dot H^s_h}
\|\wpe (t,\cdot,y_3)\|_{\dot H^s_h}.
$$
Thus we get
$$
\longformule{
\frac12 \frac d{dt} \|\wpe(t)\|_{L^{2}_{v}\dot H^{s}_{h}}^{2} +  
\|\nabla_{h}\wpe(t)\|_{L^{2}_{v}\dot H^{s}_{h}}^{2} \leq \frac14
 \|\nabla_{h}\wpe(t)\|_{L^{2}_{v}\dot H^{s}_{h}}^{2}+ C  \|\nabla_{h}  \vp^h (t)\|_{L^{\infty}_{v}L^{2}_{h}}^{2}\|\wpe(t )\|_{L^{2}_{v}\dot H^{s}_{h}}^{2}
}{{}-  \e^{2}\int_{\R} (\partial_{3}\presspun(t,\cdot,y_3) | \underline w^{\e,3}(t,\cdot,y_3))_{\dot H^{s}_{h}} dy_3
 -\int_{\R}  (\nabla^h \presspun(t,\cdot,y_3) | \underline w^{\e,h}(t,\cdot,y_3))_{\dot H^{s}_{h}} dy_3 .}
$$
By integration by parts we have, thanks to the divergence free condition on~$\wpe$,
\beno
-\int_{\R}  (\partial_3 \presspun(t,\cdot,y_3) | \wp^{\e,3}(t,\cdot,y_3))_{\dot H^{s}_{h}} dy_3 & = &
\int_{\R}  ( \presspun(t,\cdot,y_3) | \partial_3\wp^{\e,3}(t,\cdot,y_3))_{\dot H^{s}_{h}} dy_3\\
 & = &-\int_{\R}  ( \presspun(t,\cdot,y_3) | \dive_h \wp^{\e,h}(t,\cdot,y_3))_{\dot H^{s}_{h}} dy_3.\eeno
By definition of the inner product of~$\dot H^s_h$, we get
$$
-\int_{\R}  (\partial_3 \presspun(t,\cdot,y_3) | \wp^{\e,3}(t,\cdot,y_3))_{\dot H^{s}_{h}} dy_3
=\int_{\R}  ( \nabla_h\presspun(t,\cdot,y_3) |\wp^{\e,h}(t,\cdot,y_3))_{\dot H^{s}_{h}} dy_3.
$$
Thus we have
$$
\longformule{
\frac12 \frac d{dt} \|\wpe(t)\|_{L^{2}_{v}\dot H^{s}_{h}}^{2} +  
\|\nabla_{h}\wpe(t)\|_{L^{2}_{v}\dot H^{s}_{h}}^{2} \leq \frac14
 \|\nabla_{h}\wpe(t)\|_{L^{2}_{v}\dot H^{s}_{h}}^{2}
}{{}+ C  \|\nabla_{h}  \vp^h (t)\|_{L^{\infty}_{v}L^{2}_{h}}^{2}\|\wpe(t )
\|_{L^{2}_{v}\dot H^{s}_{h}}^{2} - (1-\e^2)\int_{\R}  ( \nabla_h\presspun(t,\cdot,y_3) |\wp^{\e,h}(t,\cdot,y_3))_{\dot H^{s}_{h}} dy_3.
}
$$
Now we notice that
$$-(\e^2 \partial_{3}^{2} + \Delta_{h}) \presspun = \dive (\vp^h \cdot \nabla_{h}\wpe) = 
 \dive \sum_{j = 1}^{2} \partial_{j} (\vp^j  \wpe) ,
$$
which can be written in the simpler way
$$
-(\e^2 \partial_{3}^{2} + \Delta_{h}) \presspun = \dive_{h} N^{h}
$$
with~$N^h  = \vp^h \cdot \nabla_{h}\wp^{\e,h} + \partial_{3} (\wp^{\e,3}\vp^h)$.
It is easy to check that
  for any~$\sigma \in \R$,
\begin{equation}\label{estimatepresspunhsigma}
\|\nabla_{h} \presspun\|_{L^{2}_{v}\dot H^{\sigma}_{h}} \lesssim \|N^h\|_{L^{2}_{v}\dot H^{\sigma}_{h}},
\end{equation}
simply by noticing that
\begin{eqnarray*}
\|\nabla_{h} \presspun\|_{L^{2}_{v}\dot H^{\sigma}_{h}} ^{2} & \sim & \int |\xi_{h}|^{2\sigma+2} |\widehat
\presspun(\xi)|^{2} \: d\xi \\
&\sim  &\int |\xi_{h}|^{2\sigma+4} |\widehat N^{h}(\xi)|^{2} \: \frac{d\xi}{ (|\xi_{h}|^{2}+\e^2 |\xi_{3}|^{2})^2} \\
&\lesssim& \|N^{h}\|_{L^{2}_{v}\dot H^{\sigma}_{h}} ^{2}.
\end{eqnarray*}
We infer from~(\ref{estimatepresspunhsigma}) that
\begin{eqnarray*}
\|\nabla_{h} \presspun\|_{L^{2}_{v}\dot H^{\sigma}_{h}} 
&\leq &\|\vp^h \cdot \nabla_{h}\wp^{\e,h}\|_{L^{2}_{v}\dot H^{\sigma}_{h}} 
+ \|\partial_{3} (\wp^{\e,3}\vp^h)\|_{L^{2}_{v}\dot H^{\sigma}_{h}}  \\
&\leq & \|\vp^h \cdot \nabla_{h}\wpe\|_{L^{2}_{v}\dot H^{\sigma}_{h}} 
+  \|\wp^{\e,3}\partial_{3} \vp  \|_{L^{2}_{v}\dot H^{\sigma}_{h}} + \| \vp \dive_{h} \wp^{\e,h}
\|_{L^{2}_{v}\dot H^{\sigma}_{h}}.
\end{eqnarray*}
We claim that for all~$-1<s<1$,
\ben \label{claimpresspun}
\cI_h \eqdefa \int_{\R} |(\nabla_{h} \presspun(t,\cdot,y_3) | 
\wpe(t,\cdot,y_3))|_{\dot H^{s}_{h}} dy_3 &&\nonumber\\
&&\!\!\!\!\!\!\!\!\!\!\!\!\!\!\!\!\!\!\!\!\!\!\!\!\!\!\!\!\!\!\!\!\!\!\!\!\!\!\!\!\!\!\!\!\!\!\!\!\!\!\!\!\!\!\!\!\!\!\!\!\!\!\!\!\!\!\!\!\!\!\!\!{}\leq \frac14\|\nabla_{h}\wpe \|_{L^{2}_{v}\dot H^{s}_{h}}^{2} 
+ C \| \wpe \|_{L^{2}_{v}\dot H^{s}_{h}}^{2}   \|\nabla  \vp^h \|_{L^{\infty}_{v}L^{2}_{h}}^{2}( 1 + 
  \|  \vp^h \|_{L^{\infty}_{v}L^{2}_{h}}^{2}). 
\een
Let us prove the claim. Suppose first that~$s = 0$. Then a product law gives 
\begin{eqnarray*}
\cI_h& \leq&  \| \wpe \|_{L^{2}_{v}\dot H^{\frac12}_{h}}
 \|\nabla_{h} \presspun\|_{L^{2}_{v}\dot H^{-\frac12}_{h}}\\
&\lesssim &\| \wpe \|_{L^{2}_{v}\dot H^{\frac12}_{h}}  \|  \vp^h \|_{L^{\infty}_{v}\dot H^{\frac12}_{h}}   
\|\nabla_{h}  \wpe \|_{L^{2}(\R^{3})} + \| \wpe \|_{L^{2}_{v}\dot H^{\frac12}_{h}}^{2} 
\|\partial_{3} \vp \|_{L^{\infty}_{v}L^{2}_{h}} .
\end{eqnarray*}
By interpolation, we get
$$
{\cI_h
\leq \|\nabla_{h} \wpe\|_{L^{2}(\R^{3})}^{\frac32}
 \| \wpe\|_{L^{2}(\R^{3})}^{\frac12} \|  \vp^h \|_{L^{\infty}_{v}L^2_{h}}^{\frac12} 
\|  \nabla_{h}\vp^h \|_{L^{\infty}_{v}L^2_{h}}^{\frac12}}
{
{}+  \| \wpe\|_{L^{2}(\R^{3})}  \|\nabla_{h} \wpe\|_{L^{2}(\R^{3})} \|\partial_{3} \vp \|_{L^{\infty}_{v}L^{2}_{h}}.
}
$$
The claim in the case~$s = 0$ follows from a convexity  inequality, which gives
$$
\cI_h \leq \frac14\|\nabla_{h} \wpe\|_{L^{2}(\R^{3})}^{2}
+ C  \| \wpe\|_{L^{2}(\R^{3})}^2(1+\| \vp \|_{L^{\infty}_{v}L^{2}_{h}}^2)   
\|\nabla \vp \|^2_{L^{\infty}_{v}L^{2}_{h}}.
$$
In the case when~Ê$0<s<1$, then we use the  product rule
$$
\|\nabla_{h} \presspun\|_{L^{2}_{v}\dot H^{s-1}_{h}}  \lesssim \| \vp^h \|_{L^{\infty}_{v}\dot H^{\frac12}_{h}}   
\|\nabla_{h}\wpe \|_{L^{2}_{v}\dot H^{s-\frac12}_{h}} 
 + \|\wpe\|_{L^{2}_{v}\dot H^{s}_{h}}  \|\nabla \vp \|_{L^{\infty}_{v}L^{2}_{h}}
$$
along with the fact that
$$
\cI_h \leq \|\nabla_{h} \presspun\|_{L^{2}_{v}\dot H^{s-1}_{h}} 
\|\nabla_{h}\wpe \|_{L^{2}_{v}\dot H^{s}_{h}}.
$$
Finally in the case when~Ê$-1<s<0$, we write
 \begin{eqnarray*}
\cI_h & \leq& 
\| \wpe\cdot\nabla    \vp^h\|_{L^{2}_{v}\dot H^{s}_{h}} \| \wpe \|_{L^{2}_{v}\dot H^{s}_{h}}
+   \| \wpe \cdot   \vp^h \|_{L^{2}_{v}\dot H^{s}_{h}}
 \|\nabla_{h} \wpe\|_{L^{2}_{v}\dot H^{s }_{h}}\\
&\lesssim & \| \nabla \vp^h \|_{L^{\infty}_{v}L^{2}_{h}} \| \wpe \|_{L^{2}_{v}\dot H^{s}_{h}}
\|  \nabla_{h}\wpe \|_{L^{2}_{v}
\dot H^{s}_{h}} +  \|  \vp^h \|_{L^{\infty}_{v}\dot H^{\frac12}_{h}}  \| \wpe \|_{L^{2}_{v}\dot H^{s+\frac12}_{h}} 
\|\nabla_{h} \wpe\|_{L^{2}_{v}\dot H^{s }_{h}}.
\end{eqnarray*}
The claim~(\ref{claimpresspun}) follows by interpolation. 

Using that result we obtain that
$$ 
 \frac d{dt} \|\wpe(t )\|_{L^{2}_{v}\dot H^{s}_{h}}^{2} +  
\|\nabla_{h}\wpe(t )\|_{L^{2}_{v}\dot H^{s}_{h}}^{2} \lesssim 
  \|\nabla \vp^h (t )\|_{L^{\infty}_{v}L^{2}_{h}}^{2}\|\wpe(t)\|_{L^{2}_{v}\dot H^{s}_{h}}^{2} 
(1+  \|   \vp^h (t )\|_{L^{\infty}_{v}L^{2}_{h}}^{2})
$$
and we conclude by a Gronwall lemma. Indeed we get that
$$ \longformule{
 \|\wpe(t )\|_{L^{2}_{v}\dot H^{s}_{h}}^{2} + \int_{0}^t \|
\nabla_{h}\wpe(t' )\|_{L^{2}_{v}\dot H^{s}_{h}}^{2}\: dt' \leq 
  \|w_{0}\|_{L^{2}_{v}\dot H^{s}_{h}}^{2} }{\times \exp\left(
C \int_{0}^t  \|\nabla \vp^h (t' )\|_{L^{\infty}_{v}L^{2}_{h}}^{2} (1+ 
\| \vp^h (t' )\|_{L^{\infty}_{v}L^{2}_{h}}^{2}) \: dt' 
\right).}
$$
But by the basic energy estimate~(\ref{estimenergy2Dparameter}), we have
$$
\| \vp^h (t )\|_{L^{\infty}_{v}L^{2}_{h}} \leq \|v_{0}\|_{L^{\infty}_{v}L^{2}_{h}}.
$$
Moreover, Corollary~\ref{corestimateprofil1} implies that
$$
\nabla  \vp^h \in L^{2}(\R^+;L^{\infty}_{v}L^{2}_{h}),
$$
so Lemma~\ref{lemmaestimateprofil2} is proved in the case when~Ê$\alpha = 0$. The case when~$\alpha$ is
positive is an easy adaptation of the proof of Lemma~\ref{lemmaestimateprofil1}; it is left
to the reader.
\end{proof}
Clearly Lemmas~\ref{lemmaestimateprofil1} and~\ref{lemmaestimateprofil2} allow to
 obtain Lemma~\ref{lemmaquasi2D1} stated in the previous section.


\section{The estimate of the error term}\label{estimateserror}\setcounter{equation}{0}
In this section we shall prove Lemma  \ref{lemmaquasi2D2} stated above. We need to write down precisely
 the equation satisfied by the remainder term~$R^{\e}$, and to check that the forcing terms appearing in the equation
can be made small.

Let us recall that
\beno \vapp (t,x) & = &  (( \vp^h, 0 )+\e (\wp^{\e,h},\e^{-1}\wp^{\e,3}
 ))(t,x_h,\e x_3)\and\\
  p^\e_{app}(t,x) & =  & (\presspO + \e   \presspun) (t,x_h,\e x_3).
\eeno
 It is an easy computation to see that
\begin{eqnarray*}
&&\left(\partial_{t} \vapp + \vapp \cdot \nabla \vapp - \Delta \vapp\right)(t,x_h,  x_3) = 
(\partial_{t} \vp^h+ \vp^h\cdot \nabla_{h} \vp^h - \Delta_{h} \vp^h, 0)
 (t,x_h,\e x_3) \\
&+& \e\left(\partial_{t} \wp^{\e,h} +\vp^h\cdot \nabla_{h} \wp^{\e,h} - \Delta_{h}  \wp^{\e,h}
-\e^{2} \partial_{3}^{2} \wp^{\e,h},0 \right)(t,x_h,  \e x_3) \\
&+& \left(0,\partial_{t} \wp^{\e,3} +\vp^h\cdot \nabla_{h} \wp^{\e,3} - \Delta_{h}  \wp^{\e,3}
-\e^{2} \partial_{3}^{2} \wp^{\e,3}  \right)(t,x_h,  \e x_3)+ \e \widetilde F^{\e}(t,x_h,\e x_3)
\end{eqnarray*}
where
$$ 
\widetilde  F^{\e}(t,x_h,  y_3) 
 \eqdefa \Bigl((\e\wp^{\e } \cdot \nabla\wp^{\e,h},  \wp^{\e } \cdot \nabla\wp^{\e,3})
+ (\wp^{\e} \cdot \nabla \vp^h,0)
 + \e(\partial_{3}^{2}
\vp^h,0)\Bigr)(t,x_h,  y_3) . 
$$
In order to simplify the notation, let us write~$\widetilde  F^{\e} = \widetilde  F^{\e,1}+\widetilde  F^{\e,2}$ with
\begin{eqnarray*}
\widetilde  F^{\e,1}
 &\eqdefa&
(\e\wp^{\e } \cdot \nabla\wp^{\e,h},  \wp^{\e } \cdot \nabla\wp^{\e,3})+  (\wp^{\e} \cdot \nabla \vp^h,0) \and
 \\
  \widetilde  F^{\e,2} & \eqdefa& \e(\partial_{3}^{2}
\vp^h,0) .
\end{eqnarray*}
Recalling the equations satisfied by~$\vp^h$ and~$\wp^{\e}$, we infer that
$$
\left(\partial_{t} \vapp + \vapp \cdot \nabla \vapp - \Delta \vapp\right)(t,x_h,  x_3) = -\nabla p^\e_{app} + \e G^{\e}(t,x_h,  \e x_3)
$$
with $ \ds G^{\e}(t,x_h,  y_3)  \eqdefa  \left(\widetilde  F^{\e} 
 + (0,\partial_{3} \presspO  ) \right)(t,x_h,  y_3)$ and $F^{\e} (t,x_h,  x_3) \eqdefa 
\e G^{\e} (t,x_h,\e x_3)$. Denoting~$q^\e =p^\e - p^\e_{app}$, we infer that
$$
\partial_{t} R^{\e} +  R^{\e} \cdot \nabla R^{\e} + \vapp  \cdot \nabla R^{\e}+
 R^{\e} \cdot \nabla \vapp - \Delta R^{\e} = -\nabla q^\e + F^{\e}.
$$
So Lemma~\ref{lemmaquasi2D2} will be established as soon as we prove that~$\|F^\e
\|_{L^2(\R^+;\dot H^{-\frac 1 2}(\R^3))}
 \leq C_{v_0,w_0} \e^{\frac 1 3}$.

The forcing term~$F^{\e}$ consists in three different types of terms: a pressure term  involving~$\presspO$,
a linear term~$ \e^2 \partial_{3}^{2}
\vp^h (t,x_h,\e x_3)$, and finally a number of nonlinear terms, defined as~$\e \widetilde  F^{\e,1}  (t,x_h,\e x_3)$ above.
 Each of these contributions will be dealt with separately. Let
us start by the pressure term.

\begin{lemma}\label{pressureterms}
{\sl The following estimate holds:
$$
\e \|(\partial_{3} \presspO) (t,x_{h},\e x_{3})\|_{L^2(\R^+;\dot H^{-\frac 1 2}(\R^3))}  \leq  C_{v_{0},w_{0}} 
\e^{\frac13}.
$$
}
\end{lemma}
\begin{proof}
  We define~$P_{0}^{\e} (t,x_{h},  x_{3})\eqdefa (\partial_{3} \presspO) 
 (t,x_{h},\e x_{3})$. 
Sobolev embeddings enable us to write
\begin{eqnarray*}
\|P_{0}^{\e}\|_{L^2(\R^+;\dot H^{-\frac 1 2}(\R^3))} &\lesssim& 
\|P_{0}^{\e}\|_{L^2(\R^+;L^{ \frac 3 2}(\R^3))} \\
&\lesssim& \e^{-\frac23} \|\partial_{3}\presspO\|_{L^2(\R^+;L^{ \frac 3 2}(\R^3))}.
\end{eqnarray*}
Recalling that
$$
\presspO = (-\Delta_{h})^{-1} \sum_{j,k=1}^{2}\partial_{j} \partial_{k} (\vp^j \vp^k),
$$
we have  by Sobolev embeddings,
\begin{eqnarray*}
  \|\partial_{3} \presspO\|_{L^2(\R^+;L^{ \frac 3 2}(\R^3))} &\lesssim& 
\sum_{j,k=1}^{2}\|\vp^j \partial_{3}
 \vp^k\|_{L^2(\R^+;L^{ \frac 3 2}(\R^3))} \\
 &\lesssim&  \|\vp \|_{L^\infty(\R^+;L^{3}(\R^3))} 
 \|\partial_{3}\vp \|_{L^2(\R^+;L^{3}(\R^3))} \\
 &\lesssim&  \|\vp \|_{L^2(\R^+;H^{\frac12}(\R^3))} 
 \|\partial_{3}\vp \|_{L^\infty(\R^+;H^{\frac12}(\R^3))}, 
\end{eqnarray*}
so we can conclude by Lemma~\ref{lemmaestimateprofil1}. This proves
Lemma~\ref{pressureterms}.
\end{proof}

Now let us consider the linear term~$ \e^2 \partial_{3}^{2}
\vp^h (t,x_h,\e x_3)$. The statement is the following.
\begin{lemma}\label{d32}
{\sl The following estimate holds:
$$
\e^{2} \| (\partial_{3}^2 \vp^h) (t,x_{h},\e x_{3})\|_{L^2(\R^+;\dot H^{-\frac 1 2}(\R^3))} \leq C_{v_{0}
} \e^{\frac12}.
$$
}
\end{lemma}
\begin{proof}
We have
\begin{eqnarray*}
\e^{2} \| (\partial_{3}^2 \vp^h) (t,x_{h},\e x_{3}) \|_{L^2(\R^+;\dot H^{-\frac 1 2}(\R^3))} 
&\lesssim&
\e \Bigl\|\partial_{3} \bigl (\partial_{3}  \vp^h  (t,x_{h},\e x_{3})\bigr) \Bigr\|_{L^2(\R^+;\dot H^{-\frac 1 2}(\R^3))} 
\\
&\lesssim&
\e \|(\partial_{3} \vp^h) (t,x_{h},x_{3}) \|_{L^2(\R^+;\dot H^{ \frac 1 2}(\R^3))} ,
\end{eqnarray*} 
A computation in Fourier variables  shows that, for any function~$a$ on~$\R^3$, we have
\begin{eqnarray*}
\|a(x_{h},\e x_{3}) \|_{ \dot H^{ \frac 1 2}(\R^3) } ^{2} & = &
\frac 1 {\e^{2}} \int_{\R^{3}} \Bigl |\widehat a
\Bigl(\xi_{h},\frac{\xi_{3}} \e \Bigr) \Bigr|^2 |\xi| \: d\xi \\
&\leq &\frac 1 {\e} \int_{\R^{3}} \Bigl |\widehat a
\Bigl(\xi_{h},\frac{\xi_{3}} \e \Bigr) \Bigr|^2 |\xi_h| \: d\xi_h \frac {d\xi_3} \e 
+ \int_{\R^{3}} \Bigl |\widehat a
\Bigl(\xi_{h},\frac{\xi_{3}} \e \Bigr) \Bigr|^2 \frac {|\xi_3|} \e \: d\xi_h  \frac {d\xi_3} \e \,\cdotp
\end{eqnarray*}
By interpolation, we deduce that
$$
\|a(x_{h},\e x_{3}) \|_{ \dot H^{ \frac 1 2}(\R^3) } ^{2}  \leq 
\frac1\e \|a\|_{L^{2}(\R^3)} \|\nabla_h a\|_{L^{2}(\R^3)} +
 \|a\|_{L^{2}(\R^3)}\|\partial_{3}a\|_{L^{2}(\R^3)}.
$$
 Applying this inequality with~$a=\partial_3\vp$,  we get
\begin{eqnarray*}
\e^{2} \| (\partial_{3}^2 \vp^h) (t,x_{h},\e x_{3})\|_{L^2(\R^+;\dot H^{-\frac 1 2}(\R^3))} 
&\lesssim & \e^{\frac12}\|\partial_{3} \vp (t)\|_{L^{2}(\R^{+};L^{2}(\R^3))}^{\frac12}\|\partial_{3} \nabla_{h} \vp (t)\|^{\frac 12}_{L^{2}(\R^{+};L^{2}(\R^3))}\\
&&\quad{}
+ \e \|\partial_{3}  \vp (t)\|_{L^{2}(\R^{+};L^{2}(\R^3))}^{\frac12} \|\partial_{3}^{2}
  \vp (t)\|_{L^{2}(\R^{+};L^{2}(\R^3))}^{\frac12} \\
&\leq & C_{v_{0}} \e^{\frac12}
\end{eqnarray*}
by
  Lemma~\ref{lemmaestimateprofil1}. This proves Lemma~\ref{d32}.
\end{proof}
Now let us turn to the nonlinear terms composing~$F^{\e}$, which we denoted above~$\e \widetilde  F^{\e,1}$.

\begin{lemma}\label{nonlinear}
{\sl The following estimate holds:
$$
\e  \|
\widetilde  F^{\e,1} (t,x_h,\e x_3)\|_{L^2(\R^+;\dot H^{-\frac 1 2}(\R^3))} \leq C_{v_{0},w_{0}
} \e^{\frac13}.
$$
}
\end{lemma}
\begin{proof} We recall that
$$
\widetilde  F^{\e,1}  =(\e\wp^{\e } \cdot \nabla\wp^{\e,h},  \wp^{\e } \cdot \nabla
\wp^{\e,3})+  (\wp^{\e} \cdot \nabla \vp^h,0).
$$
Notice that for all functions~$a$ and~$b$ and any~$1 \leq j \leq 3$,
\begin{eqnarray*}
\| a \partial_{j} b \|_{L^2(\R^+;\dot H^{-\frac 1 2}(\R^3))} &\lesssim &
\| a \partial_{j} b \|_{L^2(\R^+;L^{\frac 3 2}(\R^3))} 
\\
&\lesssim &\|a  \|_{L^\infty(\R^+;L^{3}(\R^3))}\| \partial_{j}
 b\|_{L^2(\R^+; L^{3}(\R^3))}.
\end{eqnarray*}
Defining~$c^{\e}(t,x_{h},x_{3})  = (a \partial_{j}b)(t,x_h,\e x_3)$ this implies
that
$$
\| c^\e \|_{L^2(\R^+;\dot H^{-\frac 1 2}(\R^3))} \lesssim \e^{-\frac23}
\|a  \|_{L^\infty(\R^+;\dot H^{ \frac 1 2}(\R^3))}\| \partial_{j}
 b\|_{L^2(\R^+;\dot H^{ \frac 1 2}(\R^3))}.
$$
We can apply that inequality to~$a$ and~$b$ equal to~$\vp $ or~$\wpe$, due to the results proved
 in Section~\ref{estimatesvapp}, and the lemma follows.
\end{proof}


\end{document}